\documentclass{article}

\usepackage[utf8]{inputenc}
\usepackage[T1]{fontenc}
\usepackage{graphicx}
\usepackage{tikz}
\usepackage{amsmath,mathtools,amsfonts,amsthm}
\usepackage[pdftex,colorlinks=true,linkcolor=blue,citecolor=blue,urlcolor=blue]{hyperref}

\newcommand\x{\mathbf{x}}
\newcommand\R{\mathbb{R}}
\newcommand\N{\mathbb{N}}
\renewcommand\div{\text{div}}
\renewcommand\d{\mathrm{d}}

\newtheorem{remark}{Remark}

\newtheorem{theorem}{Theorem}

\title{Dirichlet-Neumann and Neumann-Neumann Methods for Elliptic Control Problems}

\author{Martin Jakob Gander$^{1}$, Liu-Di LU$^{1}$}
\date
{%
\noindent{\small\textit{$^1$Section of Mathematics, University of Geneva, Rue du Conseil Général 7-9, 1205 Geneva, Switzerland}}\\
}

\begin{document}

\maketitle

\begin{abstract} 
We present the Dirichlet-Neumann (DN) and Neumann-Neumann (NN) methods applied to the optimal control problems arising from elliptic partial differential equations (PDEs) under the $H^{-1}$ regularization. We use the Lagrange multiplier approach to derive a forward-backward optimality system with the $L^2$ regularization, and a singular perturbed Poisson equation with the $H^{-1}$ regularization. The $H^{-1}$ regularization thus avoids solving a coupled bi-Laplacian problem, yet the solutions are less regular. The singular perturbed Poisson equation is then solved by using the DN and NN methods, and a detailed analysis is given both in the one-dimensional and two-dimensional case. Finally, we provide some numerical experiments with conclusions.
\end{abstract}

\section{Introduction}\label{sec:1}
Consider the state $y(\x)$ governed by the elliptic partial differential equation (PDE)
\begin{equation}\label{eq:poisson}
    -\div\left(\kappa(\x)\nabla y(\x)\right) = u(\x), \quad \x\in\Omega, \qquad y(\x)=0, \quad \x\in\partial\Omega,
\end{equation}
where $\Omega\subset\R^n$, $n=1,2,3$ is a bounded domain and $\partial\Omega$ its boundary. Here $u$ is a control variable from an admissible set $U_{\text{ad}}$, which drives the state $y$ to a target state $\hat y$. Problem~\eqref{eq:poisson} originates from the stationary heat conduction equation. In this setting, $\kappa(\x)$ denotes the thermal conductivity of $\Omega$, $y(\x)$ is the temperature at a particular position $\x$ and $u(\x)$ represents a controlled heat source. The goal is to find the optimal control variable $u^*$ which minimizes the cost functional for $\nu\in\R^+$,
\begin{equation}\label{eq:J}
    J(y,u) = \frac12\int_{\Omega} \left|y(\x) -\hat y(\x)\right|^2 \d \x + \frac{\nu}2\|u\|^2_{U_{\text{ad}}},
\end{equation}
subject to the constraint~\eqref{eq:poisson}. The term $\frac{\nu}2\|u\|^2_{U_{\text{ad}}}$ can be considered as the cost of applying such a control $u$. It is said that the control is expensive if $\nu$ is large. From a mathematical viewpoint, the presence of this term with $\nu\in\R^+$ has a regularizing effect on the optimal control.

The analysis of Domain Decomposition methods (DDMs) for the elliptic PDE~\eqref{eq:poisson} is well established, see for instance~\cite{Smith1996}. Much less is known for DD methods applied to PDE-constrained optimal control problems, see for instance~\cite{Gander2018,Heinkenschloss2005}. Although the admissible set $U_{\text{ad}}$ is often considered as $L^2(\Omega)$ for such elliptic control problems, a recent study shows that the energy space $H^{-1}(\Omega)$ can also be used for the regularization~\cite{Neumuller2021}. Moreover, this space can be expanded with $L^2(0,T;H^{-1}(\Omega))$ to treat parabolic control problems~\cite{Langer2021}. From an analytical point of view, the first-order optimality system can be simplified to a Poisson type equation by using the energy space $H^{-1}(\Omega)$, whereas a biharmonic type problem still needs to be treated for the usual $L^2(\Omega)$ regularization. Moreover, applications of the energy norm can also be found in electrical engineering, fluid mechanics~\cite{Lin2011}, etc.

Inspired by this approach, we study in this paper DDMs applied to the optimal control problem~\eqref{eq:poisson}-\eqref{eq:J} using the energy norm. More precisely, we introduce in Section~\ref{sec:2} the use of the energy norm $H^{-1}$ for the elliptic control problem, and compare the optimality system with that of the $L^2$ norm. Although we consider for simplicity an unconstrained control, this can be extended to problems with state or control constraints, see also~\cite{Troltzsch2010}. We then provide in Section~\ref{sec:3} a convergence analysis of the Dirichlet-Neumann (DN)~\cite{Bjorstad1986} and the Neumann-Neumann (NN)~\cite{Bourgat1989} methods applied to the optimality system. Some numerical experiments are given in Section~\ref{sec:4}, where we conclude with some comments.

\section{Regularization: \texorpdfstring{$L^2$}{Lg} vs \texorpdfstring{$H^{-1}$}{Lg}}\label{sec:2}
We assume that both the control $u$ and the target state $\hat y$ are in $L^2(\Omega)$, and consider first $U_{\text{ad}}=L^2(\Omega)$ as the set of all feasible controls. Using the Lagrange multiplier approach~\cite{Troltzsch2010}, we get for the first-order optimality system for problem~\eqref{eq:poisson}-\eqref{eq:J}
\begin{equation}\label{eq:opt_sys1}
    \begin{aligned}
        -\div\left(\kappa(\x)\nabla y(\x)\right) &= u(\x), \quad \x\in\Omega, &&\quad y(\x)=0, \quad \x\in\partial\Omega,\\
        -\div\left(\kappa(\x)\nabla p(\x)\right) &= y(\x) - \hat y(\x), \quad \x\in\Omega, &&\quad p(\x)=0, \quad \x\in\partial\Omega,\\
        p(\x) + \nu u(\x) &= 0, \quad \x\in\Omega,
    \end{aligned}
\end{equation}
where $p$ is the Lagrange multiplier (or adjoint state). Inserting the third equation of~\eqref{eq:opt_sys1} into the first equation, and the result into the second equation, we can rewrite the optimality system~\eqref{eq:opt_sys1} with one single variable, for instance, with respect to the state variable $y$ as
\begin{equation}\label{eq:opt_sys_red1}
    \begin{aligned}
        \nu\div\Big(\kappa(\x)\nabla\big( \div\left(\kappa(\x)\nabla y(\x)\right)\big)\Big) + y(\x) &= \hat y(\x), \quad &&\x\in\Omega, \\
        \quad \div\left(\kappa(\x)\nabla y(\x)\right) = y(\x)&=0, \quad &&\x\in\partial\Omega.
    \end{aligned}
\end{equation}
In particular, we identify in~\eqref{eq:opt_sys_red1} a biharmonic operator by taking the conductivity $\kappa(\x)=1$ everywhere over the domain.

We consider now $U_{\text{ad}}=H^{-1}(\Omega)$ in~\eqref{eq:J} as the set of all feasible controls. As proposed in~\cite{Neumuller2021}, 
we can define the norm in $H^{-1}(\Omega)$ by
\begin{equation}\label{eq:norm}
    \|u\|_{H^{-1}(\Omega)}^2:= \|\sqrt{\kappa}\nabla y\|_{L^2(\Omega)}^2,
\end{equation}
which is the energy norm. Note that the conductivity $\kappa$ is positive. On the other hand, following the same reasoning as in the $L^2(\Omega)$ case to derive the optimality system,  
we obtain
\begin{equation}\label{eq:opt_sys_red2}
    -\nu\div\left(\kappa(\x)\nabla y(\x)\right) + y(\x) = \hat y(\x), \quad \x\in\Omega, \qquad y(\x) = 0, \quad \x\in\partial\Omega.
\end{equation}
Comparing~\eqref{eq:opt_sys_red2} with the reduced optimality system under $L^2$ regularization~\eqref{eq:opt_sys_red1}, we observe that indeed only a Laplace type operator needs to be solved in~\eqref{eq:opt_sys_red2}.

\begin{remark}
    We need to be careful when comparing solutions of the two reduced optimality systems~\eqref{eq:opt_sys_red1} and~\eqref{eq:opt_sys_red2}, since we penalize the control in different norms and solve different equations. In the $L^2$ case, the control can be determined by $u = -\frac1{\nu}p$ which is proportional to the adjoint state variable, while it is proportional to the state variable in the $H^{-1}$ case, since $u=\frac1{\nu}(\hat y-y)$. Furthermore, the solution is less regular in the $H^{-1}$ case as shown in~\cite{Neumuller2021}. 
\end{remark}

\begin{remark}
    Depending on the value of $\nu$, ~\eqref{eq:opt_sys_red2} is a singularly perturbed PDE. Standard numerical methods can perform poorly, we refer to the monograph~\cite{Roos2008} for a review of robust numerical methods for such problems. In the recent work~\cite{Langer2022}, the authors use an algebraic multigrid method and a balancing domain decomposition by constraints preconditioner for a finite element discretization to treat the problem~\eqref{eq:opt_sys_red2}. They observed that optimal convergence is ensured with $\nu=h^2$, $h$ being the mesh size.  
\end{remark}

\section{Convergence Analysis of DD methods}\label{sec:3}

We now provide a convergence analysis for the DN and the NN methods applied to solve the reduced optimality system~\eqref{eq:opt_sys_red2}, and then compare with DN and NN methods applied to~\eqref{eq:opt_sys_red1} from~\cite{Gander2018}. 

Without loss of generality, the analysis is given under the assumption that the target state $\hat y=0$, meaning that we focus on the error equation related to~\eqref{eq:opt_sys_red2}. Moreover, we assume that the conductivity coefficient $\kappa(x)=1$ everywhere over the domain for the following analysis, although the DN and NN methods are defined for a general  $\kappa(x)$. Let us first consider the one-dimensional case with the domain $\Omega=(0,1)$. We decompose it into two non-overlapping subdomains $\Omega_1=(0,\alpha)$ and $\Omega_2=(\alpha,1)$ with $\alpha$ the interface. We denote by $e_i$ the error in domain $\Omega_i$ for $i=1,2$.

For the DN method, the error equations for~\eqref{eq:opt_sys_red2} are for iteration index $n=1,2,\ldots$,
\begin{equation}\label{eq:DN}
    \begin{aligned}
        & \partial_{xx} e_1^n - \nu^{-1}e_1^n = 0, \quad e_1^n(0)=0, \quad e_1^n(\alpha)=e_{\alpha}^{n-1},\\
        & \partial_{xx} e_2^n - \nu^{-1}e_2^n = 0, \quad e_2^n(1)=0, \quad \partial_x e_2^n(\alpha)=\partial_x e_1^n(\alpha),
    \end{aligned}
\end{equation}
with $e_{\alpha}^n:=(1-\theta)e_{\alpha}^{n-1} + \theta e_2^n(\alpha)$ and $\theta\in(0,1)$ a relaxation parameter. 
We notice that the error equations~\eqref{eq:DN} are similar to the ones in~\cite[Equation (2.4)]{Gander2016} for applying the Dirichlet-Neumann waveform relaxation (DNWR) method to the heat equation. Indeed, after a Laplace transform, the error equations for the DNWR method in the one dimensional case are like~\eqref{eq:DN}, where $\nu^{-1}$ is replaced by $s$. For this reason, we follow the same calculations as in~\cite{Gander2016} and find the convergence factor
\begin{equation}\label{eq:cv_DN}
    \rho_{\text{DN}}:=\left|1-\theta\left[1+\tanh\left(\sqrt{\nu^{-1}} (1-\alpha)\right)\coth\left(\sqrt{\nu^{-1}} \alpha\right)\right]\right|.
\end{equation}
This leads us to the following convergence results.

\begin{theorem}
    The DN method with $\theta=1$ applied to Problem~\eqref{eq:opt_sys_red2} converges if and only if the interface is closer to the right boundary (i.e., $\alpha > \frac12$).
\end{theorem}

\begin{proof}
    Taking $\theta=1$ in~\eqref{eq:cv_DN}, we obtain the convergence factor 
    \[\rho_{\text{DN}}=\tanh\left(\sqrt{\nu^{-1}} (1-\alpha)\right)\coth\left(\sqrt{\nu^{-1}} \alpha\right),\]
    that is smaller than 1 if and only if $\alpha > \frac12$ which can be seen by studying the function $f(x)=\sinh(1-x)\cosh(x) - \cosh(1-x)\sinh(x)$ for $x\in[0,1]$.
\end{proof}

\begin{theorem}
    For symmetric subdomains (i.e., $\alpha=\frac12$), the convergence of the DN method for Problem~\eqref{eq:opt_sys_red2} is linear and is independent of the value of the regularization parameter $\nu$. It converges in two iterations if $\theta=\frac12$.
\end{theorem}

\begin{proof}
    We just have to take $\alpha=\frac12$ in~\eqref{eq:cv_DN} and finds $\rho_{\text{DN}}=|1-2\theta|$. 
\end{proof}

\begin{theorem}\label{thm:DN_asym}
    For asymmetric subdomains (i.e., $\alpha\neq\frac12$), the DN method converges for Problem~\eqref{eq:opt_sys_red2} if and only if
    \begin{equation}\label{eq:DN_theta}
        0<\theta<2\theta_{\text{DN}}^{\star}, \quad  \theta_{\text{DN}}^{\star}:=\frac1{1+\tanh\left(\sqrt{\nu^{-1}} (1-\alpha)\right)\coth\left(\sqrt{\nu^{-1}} \alpha\right)}.
    \end{equation}
    Moreover, it converges in two iterations if and only if $\theta=\theta_{\text{DN}}^{\star}$.
\end{theorem}

\begin{proof}
    From the convergence factor~\eqref{eq:cv_DN}, the interior part of the absolute value is smaller than 1, since $\theta\in(0,1)$ and $1+\tanh\left(\sqrt{\nu^{-1}} (1-\alpha)\right)\coth\left(\sqrt{\nu^{-1}} \alpha\right)$ is strictly positive. We then just need to ensure that
    \[\theta\left[1+\tanh\left(\sqrt{\nu^{-1}} (1-\alpha)\right)\coth\left(\sqrt{\nu^{-1}} \alpha\right)\right]<2,\]
    which leads to the inequality in~\eqref{eq:DN_theta}. On the other hand, we find directly $\theta_{\text{DN}}^{\star}$ by equating~\eqref{eq:cv_DN} to zero.
\end{proof}

\begin{remark}
    As expected, we find similar results in the symmetric case as for the $L^2$ regularization. However, we have an optimal relaxation parameter for asymmetric decompositions, which is strictly smaller than 1, whereas a pair of parameters is needed for the $L^2$ regularization which can be greater than one in some cases, see~\cite{Gander2018}. This is due to the fact that two transmission conditions need to be considered for a biharmonic type problem.
\end{remark}

The error equations for the NN method, for iteration index $n=1,2,\cdots$, are
\begin{equation}\label{eq:NN_e}
    \partial_{xx} e_j^n - \nu^{-1}e_j^n = 0, \quad e_1^n(0)=0, \quad e_2^n(1)=0, \quad e_j^n(\alpha)=e_{\alpha}^{n-1}, 
\end{equation}
where the transmission condition is given by $e_{\alpha}^n:=e_{\alpha}^{n-1} - \theta\big(\psi_1^n(\alpha)+\psi_2^n(\alpha)\big)$ and $\psi_j^n$ satisfies the correction step
\begin{equation}\label{eq:NN_e_cor}
    \partial_{xx} \psi_j^n - \nu^{-1}\psi_j^n = 0, \quad \psi_1^n(0)=0, \quad \psi_2^n(1)=0, \quad \partial_{n_j} \psi_j^n(\alpha) = \partial_{n_1} e_1^n(\alpha) + \partial_{n_2} e_2^n(\alpha). 
\end{equation}
Solving~\eqref{eq:NN_e}-\eqref{eq:NN_e_cor} on each domain $\Omega_j$ and applying the boundary conditions at $x=0$ and $x=1$, we find the solutions with $A^n, B^n, C^n, D^n$ four coefficients to be determined for $e_1^n, e_2^n, \psi_1^n$ and $\psi_2^n$. Evaluating then $e_j^n$ at $x=\alpha$, and using the transmission condition $e_j^n(\alpha)=e_{\alpha}^{n-1}$, we can determine the two coefficients $A^n, B^n$ and get
\begin{equation}\label{eq:e_bis}
    e_1^n(x) = e_{\alpha}^{n-1}\frac{\sinh(\sqrt{\nu^{-1}} x)}{\sinh(\sqrt{\nu^{-1}} \alpha)}, \quad e_2^n(x) = e_{\alpha}^{n-1}\frac{\sinh\left(\sqrt{\nu^{-1}} (1-x)\right)}{\sinh\left(\sqrt{\nu^{-1}} (1-\alpha)\right) }.
\end{equation}
Similarly, we evaluate $\partial_{n_j} \psi_j^n$ at $x=\alpha$, and using the transmission condition $\partial_{n_j} \psi_j^n(\alpha) = \partial_{n_1} e_1^n(\alpha) + \partial_{n_2} e_2^n(\alpha)$ with the help of~\eqref{eq:e_bis}, we can determine the remaining two coefficients $C^n, D^n$ and get,
\[
    \begin{aligned}
        \psi_1^n(x) &= e_{\alpha}^{n-1}\frac{\sinh(\sqrt{\nu^{-1}} x)}{\cosh(\sqrt{\nu^{-1}} \alpha)} \left( \coth(\sqrt{\nu^{-1}} \alpha) + \coth(\sqrt{\nu^{-1}} (1-\alpha)) \right),\\
        \psi_2^n(x) &= e_{\alpha}^{n-1}\frac{\sinh\left(\sqrt{\nu^{-1}} (1-x)\right)}{\cosh(\sqrt{\nu^{-1}} (1-\alpha))} \left( \coth(\sqrt{\nu^{-1}} \alpha) + \coth(\sqrt{\nu^{-1}} (1-\alpha)) \right).
    \end{aligned}
\]
Using finally the definition of the transmission condition $e_{\alpha}^n$, we find the convergence factor
\begin{equation}\label{eq:cv_NN}
    \begin{aligned}
        \rho_{\text{NN}}:=\Big|1-\theta&\Big( \tanh(\sqrt{\nu^{-1}} \alpha) + \tanh\left(\sqrt{\nu^{-1}} (1-\alpha)\right) \Big) \\
        &\times \Big( \coth(\sqrt{\nu^{-1}} \alpha) + \coth(\sqrt{\nu^{-1}} (1-\alpha)) \Big)\Big|.
    \end{aligned}
\end{equation}
We obtain the following convergence results.

\begin{theorem}
    For symmetric subdomains (i.e., $\alpha=\frac12$), the convergence of the NN method for Problem~\eqref{eq:opt_sys_red2} is linear and is independent of the value of the regularization parameter $\nu$. It converges in two iterations if $\theta=\frac14$.
\end{theorem}

\begin{proof}
    We just have to take $\alpha=\frac12$ in~\eqref{eq:cv_NN} and find $\rho_{\text{NN}}=|1-4\theta|$. 
\end{proof}

\begin{theorem}\label{thm:NN_asym}
    For asymmetric subdomains (i.e., $\alpha\neq\frac12$), the NN method converges for Problem~\eqref{eq:opt_sys_red2} if and only if
    \begin{equation}\label{eq:NN_theta}
        0<\theta<2\theta_{\text{NN}}^{\star}, \quad \textstyle \theta_{\text{NN}}^{\star}:=\frac1{\Big( \tanh(\sqrt{\nu^{-1}} \alpha) + \tanh\left(\sqrt{\nu^{-1}} (1-\alpha)\right) \Big) \Big( \coth(\sqrt{\nu^{-1}} \alpha) + \coth(\sqrt{\nu^{-1}} (1-\alpha)) \Big)}.
    \end{equation}
    Furthermore, it converges in two iterations if and only if $\theta=\theta_{\text{NN}}^{\star}$.
\end{theorem}

\begin{proof}
    Following the same steps as in the proof of Theorem~\ref{thm:DN_asym}, we obtain the inequality~\eqref{eq:NN_theta}, and we find directly $\theta_{\text{NN}}^{\star}$ by equating~\eqref{eq:cv_NN} to zero.
\end{proof}

\begin{remark}
    As shown in Theorem~\ref{thm:DN_asym} and in Theorem~\ref{thm:NN_asym}, both the DN and the NN methods converge in two iterations to the exact solution. Moreover, we have a bound for the relaxation parameter $\theta$ of each method for which the convergence of the method is guaranteed.
\end{remark}

The above analysis can also be extended to the two-dimensional case. More precisely, we assume that the domain $\Omega$ is now given by $[0,1]\times[0,1]$, which is then divided into two non-overlapping subdomains $\Omega_1 = (0 , \alpha) \times [0,1]$ and $\Omega_2 = (\alpha, 1) \times  [0,1]$, with the interface at $x_1 = \alpha$ denoted by $\Gamma := \{\alpha\} \times [0,1]$. In addition, we keep the assumption that $\hat y=0$ and $\kappa(x)=1$. The two-dimensional analysis is often carried out by using a Fourier expansion in one direction, in our case, the $x_2$ direction $e_i^n(x_1,x_2)=\sum_{k=0}^{\infty}\hat e_i(x_1,k)\sin(k\pi x_2)$. In this way, the error function related to $e_i(x_1,x_2)$ passes to $\hat e_i(x_1,k)$, and for instance, in the DN case is governed by
\begin{equation}\label{eq:DN2}
    \begin{aligned}
        \partial_{x_1x_1} \hat e_1^n - \frac{\nu k^2\pi^2+1}{\nu} \hat e_1^n &= 0, \quad \hat e_1^n(0,k)=0, \quad \hat e_1^n(\alpha,k) = \hat e_{\alpha}^{n-1}, \\
        \partial_{x_1x_1} \hat e_2^n - \frac{\nu k^2\pi^2+1}{\nu} \hat e_2^n &= 0, \quad \hat e_2^n(1,k)=0, \quad \partial_{x_1} \hat e_2^n(\alpha,k) = \partial_{x_1} \hat e_1^n(\alpha,k),
    \end{aligned}
\end{equation}
with $\hat e_{\alpha}^n:=(1-\theta)\hat e_{\alpha}^{n-1} + \theta \hat e_2^n(\alpha,k)$ and $\theta\in(0,1)$. We observe that~\eqref{eq:DN2} has the same structure as in the one-dimensional case~\eqref{eq:DN}, where $\nu^{-1}$ is replaced by $\frac{\nu k^2\pi^2+1}{\nu}$. Therefore, the same type of reasoning can be applied to analyze this iteration, and we have the following results.
\begin{theorem}
    For symmetric subdomains (i.e., $\alpha=\frac12$), the convergence of the DN and the NN methods for Problem~\eqref{eq:opt_sys_red2} are both linear and independent of the value of $\nu$. It converges in two iterations if $\theta=\frac12$ for the DN method and $\theta=\frac14$ for the NN method.
\end{theorem}

\begin{theorem}
    For asymmetric subdomains (i.e., $\alpha\neq\frac12$), the DN method converges for Problem~\eqref{eq:opt_sys_red2} whenever
    \begin{equation}\label{eq:cv_DN2d}
        \rho_{\text{DN2d}}:=\sup_{k\in\N}\left|1-\theta\left[1+\tanh\left(\sqrt{\frac{\nu k^2\pi^2+1}{\nu}} (1-\alpha)\right)\coth\left(\sqrt{\frac{\nu k^2\pi^2+1}{\nu}} \alpha\right)\right]\right|<1.
    \end{equation}
    The NN method converges for Problem~\eqref{eq:opt_sys_red2} whenever
    \begin{equation}\label{eq:cv_NN2d}
        \begin{aligned}
            \rho_{\text{NN2d}}:=\sup_{k\in\N}\Big|1-\theta&\Big( \tanh(\sqrt{\frac{\nu k^2\pi^2+1}{\nu}} \alpha) + \tanh\left(\sqrt{\frac{\nu k^2\pi^2+1}{\nu}} (1-\alpha)\right) \Big) \\
            &\cdot \Big( \coth(\sqrt{\frac{\nu k^2\pi^2+1}{\nu}} \alpha) + \coth(\sqrt{\frac{\nu k^2\pi^2+1}{\nu}} (1-\alpha)) \Big)\Big|<1.
        \end{aligned}
    \end{equation}
\end{theorem}

\section{Numerical experiments}\label{sec:4}
In this section, we provide numerical experiments to illustrate the convergence rate of the DN and the NN methods for Problem~\eqref{eq:poisson}-\eqref{eq:J} with $\nu=1$ and $\hat y=0$. Figure~\ref{fig:1} (top) shows the one-dimensional convergence behaviour of these two methods for different choices of $\theta$ with an asymmetric decomposition $\alpha=\frac13$. The best choices of the relaxation parameter are given by $\theta_{\text{DN}}^{\star}\approx0.355$ and $\theta_{\text{NN}}^{\star}\approx0.229$. In particular, we observe some divergence behavior in the case of the NN method for $\theta=0.5$ and $\theta=0.7$. Indeed, this corresponds to the result in Theorem~\ref{thm:NN_asym}, since these two values are greater than $2\theta_{\text{NN}}^{\star}$ which is the upper bound for the relaxation parameter $\theta$. Furthermore, we observe the convergence to the exact solution in two iterations for a non-symmetric domain decomposition, whereas a three-step convergence is needed for the $L^2$ regularization~\cite{Gander2018}. Figure~\ref{fig:1} (bottom) presents the behavior of the convergence factors~\eqref{eq:cv_DN2d} and~\eqref{eq:cv_NN2d} in the two-dimensional case. The interface here is chosen to be asymmetric $\Gamma=\{\frac13\}\times[0,1]$. We observe good convergence behaviors for some tested relaxation parameters $\theta$. Furthermore, the NN method does not converge for $\theta=0.5$ and $\theta=0.7$ as in the one-dimensional case. We obtain that $\rho_{\text{DN2d}}\approx0.173$ for $\theta_{\text{DN2d}}^{\star}\approx0.414$ and $\rho_{\text{NN2d}}\approx0.046$ for $\theta_{\text{NN2d}}^{\star}\approx0.239$. These two optimal relaxation parameters can also be found by equioscillating the value of the convergence factor both at $k=0$ and $k\rightarrow \infty$. Moreover for each method, we find that these optimal relaxation parameters stay very close between the one-dimensional and the two-dimensional case.
\begin{figure}
    \centering
    \includegraphics[scale=0.12]{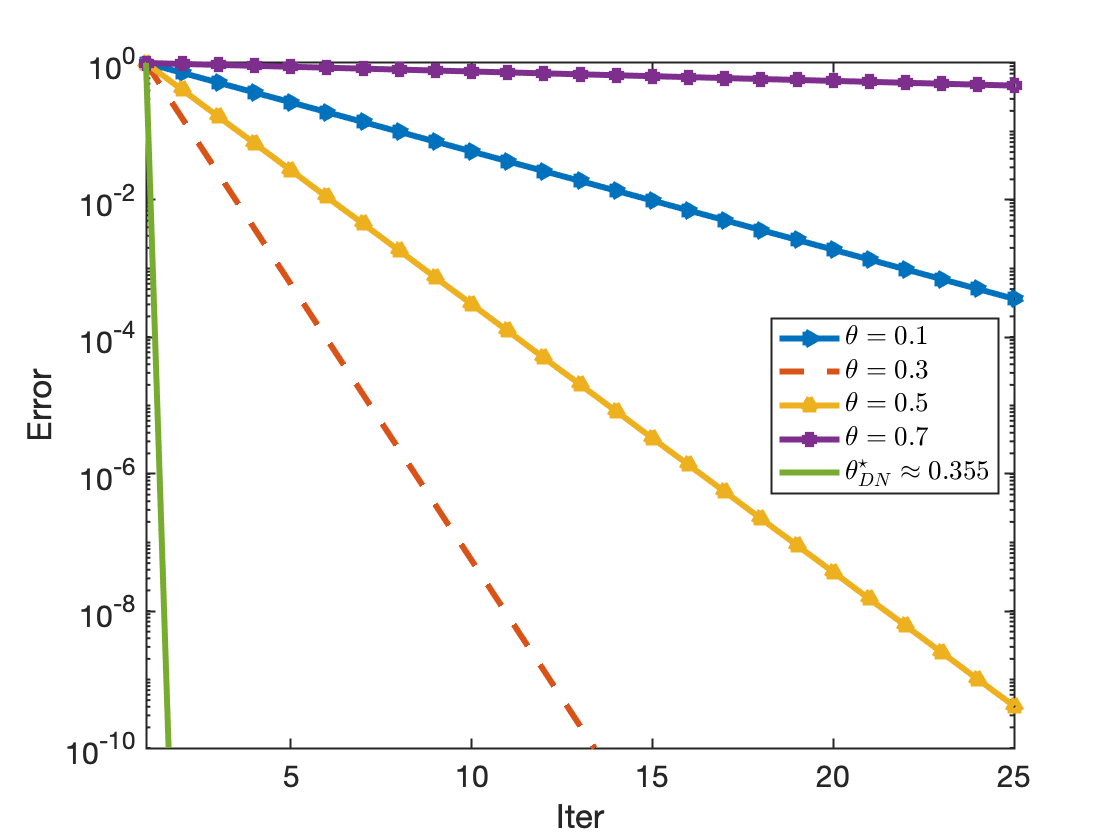}
    \includegraphics[scale=0.12]{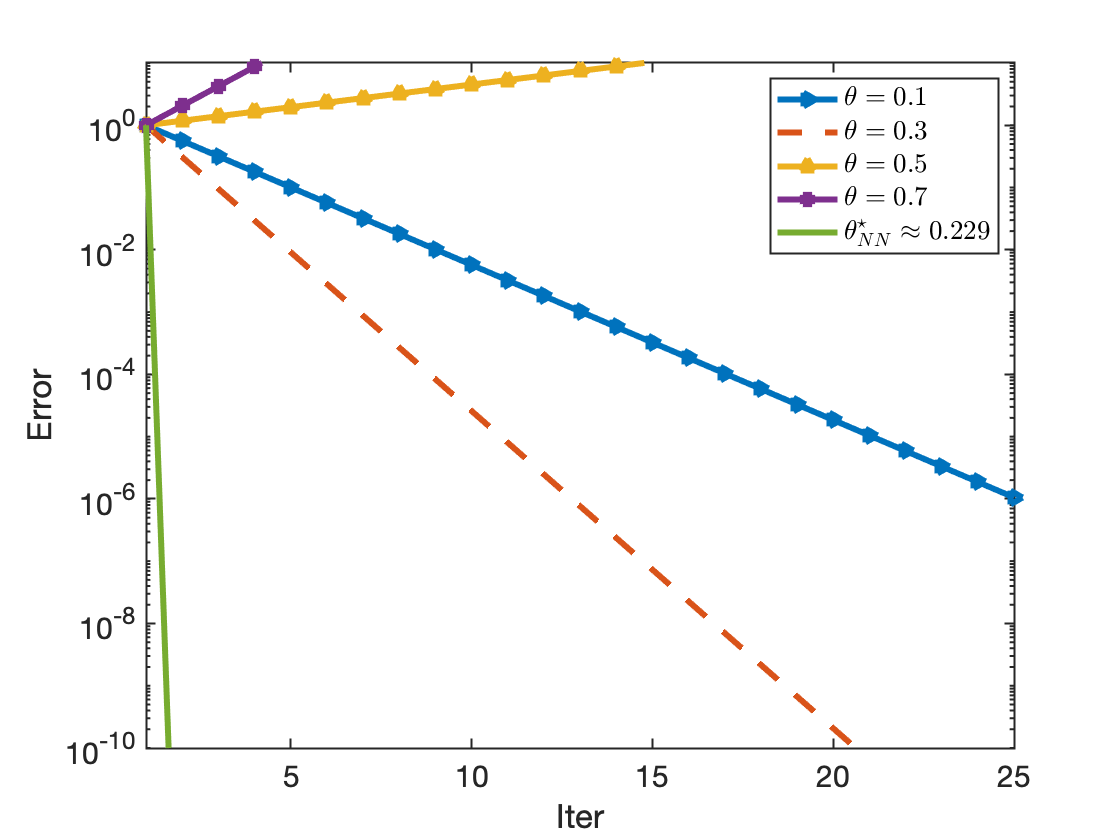}
    \includegraphics[scale=0.12]{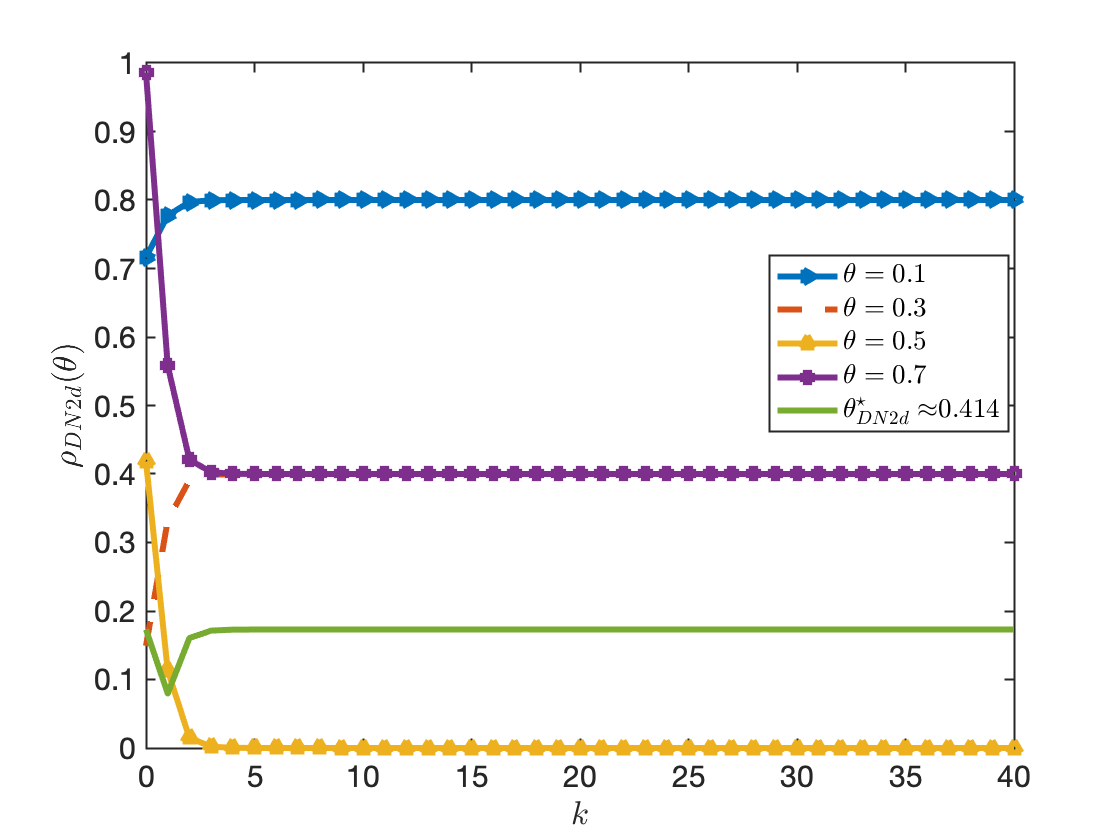}
    \includegraphics[scale=0.12]{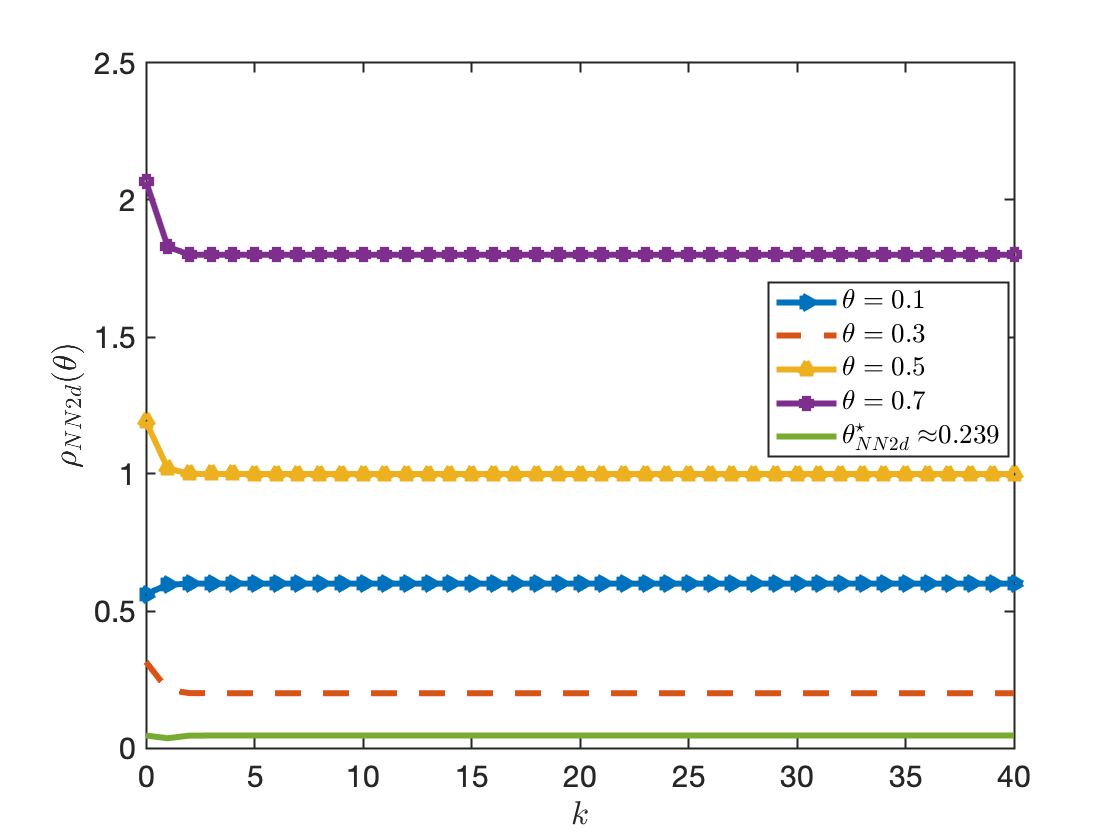}
    \caption{Error decay in 1D w.r.t. the number of iterations for the DN method (top-left) and the NN method (top-right) with the interface at $\alpha=\frac13$. Convergence factors~\eqref{eq:cv_DN2d} and~\eqref{eq:cv_NN2d} in 2D w.r.t. the value of $k\in[0,40]$ for the DN method (bottom-left) and the NN method (bottom-right) with the interface at $\Gamma=\{\frac13\}\times[0,1]$.}
    \label{fig:1}
\end{figure}

To conclude, we presented a convergence analysis of the DN and the NN methods for elliptic optimal control problems using the energy norm for regularization. Only one Poisson type equation needs to be solved, whereas a biharmonic type equation is required for $L^2$ regularization. Under the energy norm, we found similar results in the symmetric case as for the Poisson problem. Therefore, we can expect similar convergence behavior for many subdomains as presented in~\cite{Chaouqui2018}. Furthermore, explicit formulations along with an upper bound are also given for the optimal relaxation parameters with a non-symmetric decomposition, for which the methods converge still in two iterations in the one-dimensional case.
\bibliography{biblio}
\bibliographystyle{plain}
\end{document}